\documentclass[reqno,11pt]{amsart}
\usepackage{color}

 \usepackage{mathtools}

\newtheorem{theorem}{Theorem}[section]
\newtheorem{lemma}[theorem]{Lemma}
\newtheorem{corollary}[theorem]{Corollary}

\theoremstyle{definition}
\newtheorem{assumption}[theorem]{Assumption}

\theoremstyle{remark}
\newtheorem{remark}[theorem]{Remark}

 \makeatletter 
 \def\dashint{\operatorname{\,\,\,\mathclap{\int} \kern-.23em\text{\bf--}\!\!}}

\def\dashnorm{\,\,\text{\bf--}\kern-.5em\|}
\def\ninf{\qopname\relax\@empty{inf\phantom{p}\!\!\!}}

 \makeatother

\def\sft{{\sf t}}

\newcommand\bC{\mathbb{C}}

\newcommand\bR{\mathbb{R}}
\newcommand\bS{\mathbb{S}}

\newcommand\cB{\mathcal{B}}

\newcommand\cF{\mathcal{F}}

\newcommand \cL{\mathcal{L}}

\newcommand\cN{\mathcal{N}}
 
\newcommand\cS{\mathcal{S}}

\newcommand\frT{\mathfrak{T}}

\newcommand{\tr}{{\rm tr}\,}

\newcommand{\loc}{{\rm loc}\,}

 \newcommand{\mysection}[1]{\section{#1}
 \setcounter{equation}{0}}

\newcommand{\esssup}{\operatornamewithlimits{ {ess\,sup}}}

\def\bLpqloc{$b\in L_{(p,q),\loc}$}

\begin{document}

\title{Diffusion processes with \protect\bLpqloc}
\author[]{N.V. Krylov}
\address{School of Mathematics, University of Minnesota, Minneapolis, MN, 55455}
\email{nkrylov@umn.edu}



\begin{abstract}
 We present some conditions sufficient to
guarantee the existence of a 
strong Markov diffusion process
with drift $b$ and uniformly nondegenerate
bounded matrix-valued diffusion in case
$b\in L_{(p,q),\loc}$. An example
is given to illustrate how close these conditions are to be necessary.
\end{abstract}

\maketitle

\mysection{Introduction}

Let $\bR^{d}$   be a Euclidean space of
points $x=(x^{1},...,x^{d})$ and set
$\bR^{d+1}=\{(t,x):t\in\bR,x\in\bR^{d}\}$,
$$
D_{i}=\frac{\partial}{\partial x^{i}},\quad
D_{ij}=D_{i}D_{j},\quad \partial_{t}=
\frac{\partial}{\partial t}.
$$
The goal of this article is to present some new
results guaranteeing the existence of a strong
Markov diffusion process corresponding to the  
parabolic operator
\begin{equation}
                           \label{6.10.1}
\cL u(t,x)=\partial_{t}u(t,x)+(1/2)a^{ij}(t,x)D_{ij}u(t,x)+b^{i}(t,x)D_{i}u(t,x), 
\end{equation}
where 

(i) $a(t,x)=(a^{ij}(t,x))$ is a Borel
function on $\bR^{d+1}$ 
with values in the set  $\bS_{\delta}$, which is the set of  $d\times d$ symmetric
matrices  whose eigenvalues are in $[\delta,
\delta^{-1}]$ and $\delta\in(0,1]$ is a number
fixed throughout the paper,

(ii) $b(t,x)=(b^{i}(t,x))$ is a Borel
$\bR^{d}$-valued 
function on $\bR^{d+1}$. 

The connection between the theory of
elliptic and parabolic second-order
equations and the theory of Markov diffusion
processes is very well established and revealed
by many authors like A.N. Kolmogorov, I.I. Gikhman, 
E.B. Dynkin, K. It\^o to name a few. K. It\^o invented
a powerful tool of making this connection even deeper by introducing It\^o stochastic equations allowing one to view the trajectories of
the Markov diffusion processes as solutions
of It\^o's equations. He himself considered
equations only with Lipschitz coefficients.
Skorokhod made a big leap forward coming up
with a general idea which allowed him to treat
It\^o's equations with only continuous coefficients without any nondegeneracy condition on $a$. The author used his method
and proved the solvability of It\^o's equations
for {\em measurable bounded\/} $a,b$ when $a$ is uniformly nondegenerate. Generally, for such coefficients there is no uniqueness of
distributions of  solutions coming from the same initial point.  Still, it turns out that
one can choose solutions starting from any
initial point in such a way that the resulting
process becomes strong Markov, see, for instance,
 \cite{Kr_73_1},  \cite{AP_77}, \cite{GM_01}, \cite{Kr_25}.
A different approach to constructing
strong Markov diffusion processes is suggested
in \cite{Ba_98} where the drift is bounded unlike
\cite{GM_01}, \cite{Kr_25}, where the drift could be unbounded. It seems that in the case of $b$ belonging to Lebesgue spaces and bounded measurable 
uniformly nondegenerate $a$,
the widest class of coefficients
for which the strong Markov diffusion processes are  constructed
so far is found in \cite{Kr_25}, \cite{Kr_26},
although in \cite{AP_77} the authors treat
the operators $\cL$ containing integro-differential part, so  that the corresponding
Markov process has jumps.

It is  probably  worth noting the area of research in
the theory of It\^o's equation when all solutions (perhaps distinct)
starting from the same point have the same
distribution. In that case the solutions are
automatically Markov processes. The classical
source of such results is \cite{SV_79}
in case $a,b$ bounded and $a$ is uniformly nondegenerate and is continuous
in $x$ uniformly with respect to $t$. In
case $a=(\delta^{ij})$ this result was recently essentially covered in \cite{RZ_20}
where   $b$ is allowed to satisfy the Ladyzhenskaya-Prodi-Serrin condition
instead of just being bounded as in \cite{SV_79}. This condition
is substantially weekend in
 the series \cite{Ki_24}, \cite{Ki_25},
\cite{KM_22} where $b$ is in a Morrey class. Further results in this direction, which
is tangential to the main subject of the
present paper, can be found
in \cite{Ga_24},
 \cite{XXZZ_20} (where $a$ is continuous in $x$), and \cite{Kr_26}.

For  $p,q\in(1,\infty)$ denote by $L_{(p,q)}$ the space of functions $f(t,x)$
on $\bR^{d+1}$ such that
$$
\|f\|_{L_{(p,q)}}=\begin{cases}
\Big(\int_{\bR}
\Big(\int_{\bR^{d}}|f(t,x)|^{p}\,dx\Big)^{q/p}
\,dt\Big)^{1/q}<\infty\quad\text{if}\quad p\geq q;\\
\Big(\int_{\bR^{d}}
\Big(\int_{\bR}|f(t,x)|^{q}\,dt\Big)^{p/q}
\,dx\Big)^{1/p}<\infty\quad\text{if}\quad p\leq q.
\end{cases}
$$
By $L_{(p,q)}(\Gamma)$ we mean the set
of $f$ on $\Gamma$ such that $fI_{\Gamma}\in  L_{(p,q)}$.

Let $\Omega$ be the set of $\bR^{d+1}$-valued
 continuous function $(t_{0}+t,x_{t})$, $t_{0}\in \bR$,
defined for $t\in[0,\infty)$.
For $\omega=\{(t_{0}+t,x_{t}),t\geq0 \}$, define
$\sft_{t}(\omega)=t_{0}+t$, $x_{t}(\omega)=x_{t}$,
and set $ \cN_{t}=\sigma((\sft_{s},x_{s}),s\leq t)$,
$  \cN_{\infty}= \sigma((\sft_{s},x_{s}),s< \infty)$. Denote by $\frT$ the set of {\em bounded\/} stopping times
relative to $\{\cN_t\}$.  

Assume that for each $(t,x)\in\bR^{d+1}$ we are
given a probability measure $P_{t,x}$ on
$(\Omega,\cN_{\infty})$. Recall that  the set
$X=((\sft_{\cdot},x_{\cdot}),\cN_{t},P_{t,x})$
is 
a strong Markov process if

(i) for each $A\in \cN_{\infty}$ the function
$P_{t,x}(A)$ is Borel, and $P_{t,x}\big((\sft_{0},x_{0})=(t,x)\big)=1$.

(ii) for each $\tau\in\frT$, $A\in\cN_{\tau}$,
$s\geq0$, Borel $\Gamma\subset \bR^{d+1}$,
$(t,x)\in\bR^{d+1}$ we have
$$
P_{t,x}(A,(\sft_{\tau+s},x_{\tau+s})\in\Gamma)
=E_{t,x}I_{A}P_{\sft_{\tau },x_{\tau }}(
(\sft_{s},x_{s}\in \Gamma),
$$
where $E_{t,x}$ is the symbol of the expectation with respect to $P_{t,x}$.

We call $X$ a {\em strong Markov $\bR^{d+1}$-valued diffusion process
corresponding to $\cL$ (or to $a,b$)\/} if,
for any $(t,x)\in \bR^{d+1}$,
\begin{equation}
                                                     \label{4.27.5}
P_{t,x}\Big(
 \int_{0}^{T}|b(\sft_{s},x_{s})|\,ds<\infty,\quad\forall T<\infty\Big)=1
\end{equation}
and the process
$$
\eta_{t}(u)=u(\sft_{t},x_{t})-\int_{0}^{t}\cL u(\sft_{s},x_{s})\,ds-u(\sft_{0},x_{0})
$$
is a local martingale relative to $\{\cN_{t}\}$ for all $u\in C^{\infty}
_{0}(\bR^{d+1})$. In the terminology of
Stroock-Varadhan $X$ is a solution of martingale
problem.

\begin{remark}
                         \label{remark 8.5.1}

 Owing to 
Lemma 3.4.1 of \cite{Kr_25},  if $X$ is a  strong Markov $\bR^{d+1}$-valued diffusion process
corresponding to $\cL$, then
for any $(t_{0},x_{0})\in\bR^{d+1}$
there exists a $d$-dimensional Wiener process $w_{t}$, $t\geq0$,
which is a Wiener process relative to $\bar \cN_{t}$,
where $\bar \cN_{t}$ is the completion
 of $\cN_{t}$
with respect to all $P_{s,y}$, and such that with 
$P_{t_{0},x_{0}}$-probability one, for
all $s\geq 0$  
\begin{equation} 
                             \label{4.27.10}
x_{s}=x_{0}+\int_{0}^{s}\sigma(t_{0}+u,x_{u})\,dw_{u}
+\int_{0}^{s}b(t_{0}+u,x_{u})\,du,\quad \sft_{s}=t_{0}+s,
\end{equation}
where $\sigma=\sqrt a$.
\end{remark}
Here is part of
Theorem 1.9.1 of \cite{Kr_26}. 

\begin{theorem}
                    \label{theorem 6.19.10}
Let $p,q $ satisfy
\begin{equation}
                                       \label{9.20.1}
p,q\in(1,\infty),\quad \frac{d}{p}+\frac{1}{q}\leq 1.
\end{equation}
Suppose that $b\in L_{(p,q)}$. Then there exists a strong Markov $\bR^{d+1}$-valued diffusion process
$X =((\sft_{t},x_{t}),\cN_{t},P _{t,x})$  corresponding to $a,b$. 
\end{theorem}

Our goal in this article is to localize
the condition that $b\in L_{(p,q)}$.
Note that this condition does not allow $b$
to be just bounded, but admits bounded $b$ with compact support, which however is enough for many
applications.

If $\Gamma$ is a Borel subset of $\bR^{d}$
or $\bR$ by $|\Gamma|$ we mean its Lebesgue
measure and for suitable $f$ denote
$$
\dashint_{\Gamma}f\,dx=\frac{1}{|\Gamma|}
\int_{\Gamma}f\,dx.
$$
If $\Gamma=(S,T)\times B$, where $B$ is a ball
in $\bR^{d}$, and $p\geq q$ we set
$$
\dashnorm f\|_{L_{(p,q)}(\Gamma)}=
\Big(\dashint_{(S,T)}\Big(\dashint_{B}|f|^{p}
\,dx\Big)^{q/p}\,dt\Big)^{1/q}.
$$
If $p<q$,  the order of integration is,
naturally, reversed in this definition.

Fix numbers
\begin{equation}
                              \label{10.7.1}
\lambda_{0}\geq
(\delta^{-1}d)\vee \ln 6 ,\quad\nu_{0}=11/12.
\end{equation}

For $ \rho>0$ and $(t,x)\in\bR^{d+1}$
define $C_{ \rho}(t,x)=[t,t+\rho^{2})\times
B_{\rho}(x)$, where $B_{\rho}(x)=\{y\in\bR^{d}:
|y-x|<\rho\}$, $C_{\rho}=C_{\rho}(0,0)$,
$B_{\rho}=B_{\rho}(0)$. Let $\bC_{ \rho}$ be
 the collection of $C_{ \rho}(t,x)$, $\bC=\{\bC_{ \rho},\rho>0\}$. If $C =[t,t+\rho^{2})\times
B_{\rho}(x)$, function $f(s,y)$ and $\lambda\geq0$ we set 
$$
f^{(\lambda) }_{C}(s,y)
=f(s,y)e^{-\lambda (s-t)},\quad 
\dashnorm f \|_{L^{\lambda}_{(p,q)}(C)}=
\dashnorm f^{(\lambda)}_{C}\|_{L_{(p,q)}(C)}
$$
$$
\hat b^{(\lambda)}_{\rho}:=\sup_{C\in \bC_{ \rho}}
\dashnorm b \|_{L^{\lambda}_{(p,q)}(C)}.
$$

Here is our first main result
proved in Section \ref{section 9.12.1}. 
\begin{theorem}
                  \label{theorem 6.19.1}
Let $p,q$ satisfy \eqref{9.20.1}. Then
there exists
$\hat b>0$, depending only on $d,\delta,p,q$,
such that if, for some  $\rho_{0}\in[0,\infty)$,
\begin{equation}
                             \label{7.18.1}
\hat b^{(\lambda_{0}/\rho_{0}^{2})}_{\rho_{0}}\leq 
\hat b \rho_{0}^{ -1},
\end{equation}
 then
there exists a strong Markov $\bR^{d+1}$-valued diffusion process
$X=((\sft_{t},x_{t}),\\ \cN_{t},P_{t,x})$   corresponding to $a,b$.
\end{theorem}

\begin{remark}
                         \label{remark 9.30.1}
Obviously $\hat b^{(\lambda)}_{\rho}\leq \hat b^{(0)}_{\rho}$. Therefore \eqref{7.18.1}
is satisfied if
$$
\sup_{C\in \bC_{\rho_{0}}}
\dashnorm b \|_{L _{(p,q)}(C)}\leq 
\hat b  \rho_{0}^{ -1}.
$$
Also, if $\dashnorm b \|_{L _{(p,q)}(C)}<\infty$ for a $C\in\bC_{\rho}$, the dominated
convergence theorem guarantees that there is a $\lambda >0$ such that
$\dashnorm b \|_{L^{\lambda} _{(p,q)}(C)}\leq\hat b\rho^{-1}$. So, the point of \eqref{7.18.1} is that this should hold in some 
uniform way.
\end{remark}

\begin{remark}
                         \label{remark 8.5.2}
If $b$ is bounded, to satisfy \eqref{7.18.1}
with any given $\hat b$  
it suffices to take  $\rho_{0}=1$
and $\lambda_{0}$ large enough.
In the time-homogeneous case 
with bounded $b$ Theorem
\ref{theorem 6.19.1}  was obtained in \cite{Kr_73_1}.  
If $b$ is bounded, but there also
jumps, the result is found in \cite{AP_77}
and if $p=q$ in \cite{GM_01}.

Also, if $b\in L_{(p,q)}$
(as in Theorem \ref{theorem 6.19.10}), to satisfy \eqref{7.18.1}
with any given $\hat b$  
observe that for $C\in \bC_{\rho}$
$$
\dashnorm b \|_{L_{(p,q)}(C)}=N(d,p,q)\|
b ||_{L_{(p,q)}(C)}\rho^{-d/p-2/q}.
$$
If $d/p+2/q>1$, it suffices to take
$\rho_{0}$ big enough, whereas, if $d/p+2/q\leq1$, it suffices to take
$\rho_{0}$ small enough when, even if
$d/p+2/q=1$, \eqref{7.18.1} is valid because
$$
\lim_{\rho\downarrow0}\sup_{C\in\bC_{\rho}} \|b\|_{L_{(p,q)}(C)}= 0.
$$
\end{remark}

\begin{remark}
                        \label{remark 9.12.1}
Generally, the process $X$ is not unique
or {\em Feller\/}
even if $a$ is the unit matrix (see, for instance, \cite{Kr_26_1}).

\end{remark}

Next, we relax condition \eqref{7.18.1}
to allow $|b|$ to grow ``linearly'' as
$|x|\to\infty$. 

Fix a constant $K_{0}\in(0,1]$ and set
$$
\hat\rho_{0}(|x|)=\begin{cases}K_{0}\quad\text{for}
\quad |x|\leq e,\\
K_{0}(|x|\ln |x|)^{-1} \quad\text{for}
\quad |x|> e.
\end{cases}
$$
\begin{theorem}
                  \label{theorem 6.19.01}
                  
Let $p,q$ satisfy \eqref{9.20.1}.
Then there exists
$\hat b>0$, depending only on $d,\delta,p,q$,
such that if,  for all  $(t,x)\in\bR^{d+1}$
and a finite function $\hat \rho=\hat\rho(|x|)$, satisfying
\begin{equation}
                           \label{8.15.6}
  \hat\rho(|x|)\geq \hat\rho_{0}(|x|),
\end{equation}
we have
\begin{equation}
                           \label{8.13.1}
\dashnorm b\|_{L^{\lambda_{0}/\hat\rho^{2}}_{(p,q)}(C_{  
\hat\rho}(t,x))}\leq \hat b \hat\rho^{-1},
\end{equation}
where $\hat\rho =\hat\rho(|x|)$, 
 then
there exists a strong Markov $\bR^{d+1}$-valued diffusion process
$X=((\sft_{t},x_{t}),\cN_{t},P_{t,x})$   corresponding to $a,b$.
\end{theorem}

This theorem is proved in Section
\ref{section 8.16.1}.
 
\begin{remark}
                          \label{remark 8.13.3}
If
$|b(t,x)|\leq K(1+|x|)$, then,
for $\hat\rho(|x|)= eK_{0}/(|x|\vee e)$
$$
\dashnorm b\|_{L_{(p,q)}(C_{ \hat
\rho }(t,x))}\leq K(1+|x|+\hat\rho)    
\leq  \hat\rho^{-1}A,
$$
where
$$
A=\frac{eK_{0}}{|x|\vee e}(K+|x|+K_{0})\leq K_{0}e
(K+2),
$$
which can be made less than $\hat b$ by choosing $K_{0}$ small enough.

Condition \eqref{8.13.1} coincides with \eqref{7.18.1} if $\hat\rho\equiv1$.
 
\end{remark}

\begin{remark}
                          \label{remark 8.20.1}
It is easy to see from the proof of
Theorem \ref{theorem 6.19.01} that on can
take $\hat\rho_{0}(|x|)=K_{0}(|x|(\ln |x|)
\ln\ln |x|)^{-1}$ for $|x|\geq 40$.
\end{remark}

\begin{remark}
                       \label{remark 8.30.1}
It is almost obvious that it suffices
to assume that \eqref{8.13.1}
holds only for $|x|\geq R$ with any fixed $R$.

\end{remark}

Next theorem addresses the same issues but
in case of drift   decomposed into the sum
when each term in the sum may have different
$q,p$ attached. 

Let $b_{j}(t,x)=(b^{i}_{j}(t,x)),j=1,...,n,$ be Borel
$\bR^{d}$-valued 
functions on $\bR^{d+1}$, such that $b=b_{1}+...+b_{n}$,
and let $p_{j},q_{j}\in(1,\infty)$ satisfy
$$
\frac{d}{p_{j}}+\frac{1}{q_{j}}\leq 1.
$$
For $\lambda\geq0$ 
introduce
$$
\check b^{(\lambda)}_{\rho}(x):=\sup_{t}\sum_{j=1}^{n}\dashnorm b_{j}\|_{L^{\lambda }_{(p_{j},q_{j})}(C_{  
 \rho}(t,x))}.
$$
\begin{theorem}
                                   \label{theorem 9.19.1}
Assume that $b$ is decomposed as described above.
Then there exists
$\hat b>0$, depending only on $d,\delta,p_{j},q_{j}$,
such that if,  for all  $(t,x)\in\bR^{d+1}$
and a finite function $\hat \rho=\hat\rho(|x|)$, satisfying
\eqref{8.15.6}, we have
\begin{equation}
                           \label{8.13.10}
\check b_{\hat\rho(|x|)}
^{(\lambda_{0}/\hat\rho^{2}(|x|))}( x)\leq \hat b \hat\rho^{-1}(|x|),
\end{equation}
 then
there exists a strong Markov $\bR^{d+1}$-valued diffusion process
$X=((\sft_{t},x_{t}),\cN_{t},P_{t,x})$   corresponding to $a,b$.

\end{theorem}

This theorem is proved in Section
\ref{section 8.16.1}. Here we only note that,
obviously, $\check b^{(\lambda)}_{\rho}(x)\leq
\check b^{(0)}_{\rho}(x)
\leq e^{\lambda\rho^{2}}\check b^{(\lambda)}_{\rho}(x)$ and, if $b$ is independent of $t$,
then $\check b^{(\lambda)}_{\rho}(x)=
\check b^{(0)}_{\rho}(x)\big[(1-e^{-\lambda q\rho^{2}})
/(\lambda q\rho^{2})\big]^{1/q}$.
In that case if $\check b^{(0)}_{\rho}(x)$
is a bounded function of $x$, for any $\hat b$ one can find large
$\lambda_{0}$ such that \eqref{8.13.10}
holds with $\hat\rho(x)\equiv 1$. However,
what we get from Theorem \ref{theorem 9.19.1}
in this case
is weaker than Theorem 3.5.6 of \cite{Kr_25}
because here we have $p>d$ and there $p=d$.
In this connection a very challenging problem
arises: is it true that if $\check b^{(0)}_{1}(x)$ in the general case is bounded, then 
there exists a strong Markov $\bR^{d+1}$-valued diffusion process
$X=((\sft_{t},x_{t}),\cN_{t},P_{t,x})$   corresponding to $a,b$?

\mysection{Some properties of solutions
of \eqref{4.27.10}}
                     \label{section 8.2.1}

The estimates in this section are needed
to show that the solutions of \eqref{4.27.10}
do not blow up in finite time.  
Let $(\Omega,\cF,P)$ be a complete probability
space, $\{\cF_{t},t\geq0\}$ be an increasing
filtration of complete $\sigma$-fields
$\cF_{t}\subset \cF$.
Let $(w_{t},\cF_{t}), t\geq0$, be a $d$-dimensional
Wiener process and $x_{t},t\geq 0$, be a
$d$-dimensional continuous $\cF_{t}$-adapted
process. We fix $(t_{0},x_{0})\in\bR^{d+1}$  
and we are going to deal 
with the process $x_{s}$ 
which satisfies \eqref{4.27.10} 
assuming that such $x_{t}$ exists.
Set $\sft_{s}=t_{0}+s$. It is convenient to state
some results in the setting of Theorem \ref{theorem 9.19.1}.

\begin{lemma}
                         \label{lemma 7.27.2}
  Let $\lambda,\rho>0$,
$\tau$ be a stopping time,   $y=y(\omega)$  be  $\bR^{d}$-valued
$\cF_{\tau}$-measurable. Suppose that in the setting of
Theorem \ref{theorem 9.19.1} we have $\check b^{(0)}_{\rho}(y)<\infty$
(a.s.)
Then on the set
$\{\tau<\infty\}$ it holds that
\begin{equation}
                             \label{7.27.3}
E_{\cF_{\tau}}\int_{0}^{\tau_{\rho,y}}e^{-\lambda s}|b(\sft_{\tau+s},x_{\tau+s})|\,ds 
\leq N(d,\delta,p,q,\lambda, \rho, \check b^{(0)}_{\rho}(y(\omega))),
\end{equation}
with probability one,
where $\tau_{\rho,y}$ is the first exit time
of $ x_{\tau+s} $ from $ B_{\rho}(y)$ and
$$
E_{\cF_{\tau}}\Big(...\Big)=
E \Big\{\Big(...\Big)\mid \cF_{\tau}\Big\}.
$$
\end{lemma}

Proof. It is not hard to see that, when $y$
is an independent deterministic variable,
$$
\int_{0}^{\tau_{\rho,y}}e^{-\lambda s}|b(\sft_{\tau+s},x_{\tau+s})|\,ds
$$
 is jointly measurable with respect to $(\omega,y)$. It follows that, it suffices
to prove \eqref{7.27.3} for each fixed $y$.
For $y$ fixed
it suffices to prove that \eqref{7.27.3}
holds for any $T\in(0,\infty)$ with
$$
u_{\tau,T}:=E_{\cF_{\tau}}\int_{0}^{\tau_{\rho,y}\wedge T}e^{-\lambda s}|b(\sft_{\tau+s},x_{\tau+s})|\,ds 
$$
in place of its left-hand side. Owing to 
Theorem 1.1.12 and Remark 1.1.14 of \cite{Kr_26}, for each $T$,
$u_{T}$ is bounded by a constant independent
of $\tau$. Let
$$
M_{T}=\sup_{\tau}\esssup u_{\tau,T}.
$$
Clearly, $M_{T}$ is an increasing function and  for $T>\rho^{2}$
$$
u_{T} = E_{\cF_{\tau}}I_{\tau_{\rho,y}\leq \rho^{2}}
\int_{0}^{\tau_{\rho,y}\wedge \rho^{2}}e^{-\lambda s}
|b(\sft_{\tau+s},x_{\tau+s})|\,ds
$$
\begin{equation}
                        \label{8.1.2}
+e^{-\lambda \rho^{2}}E_{\cF_{\tau}}I_{\tau_{\rho,y}>\rho^{2}}
E_{ _{\cF_{\tau +\rho^{2}}}}\int_{0}^{(\tau_{\rho,y}-\rho^{2})\wedge (T-\rho^{2})}e^{-\lambda s}|b(\sft_{\tau+\rho^{2}+s},x_{\tau+\rho^{2}+s})|\,ds.
\end{equation}
Here on the set $\{\tau_{\rho,y}>\rho^{2}\}$ ($\in
\cF_{\tau+\rho^{2}}$), $\tau_{\rho,y}-\rho^{2} $
is the first exit time of $x_{\tau+\rho^{2}+s}$
from $B_{\rho}(y)$ and $T-\rho^{2}<T$. Therefore,
by definition the interior conditional expectation is less than $M_{T}$.

By Theorem 1.1.12 and Remark 1.1.14 of \cite{Kr_26}
$$
E_{\cF_{\tau}}I_{\tau_{\rho,y}\leq \rho^{2}}
\int_{0}^{\tau_{\rho,y}\wedge \rho^{2}}e^{-\lambda s}
|b(\sft_{\tau+s},x_{\tau+s})|\,ds
$$
\begin{equation}
                                        \label{9.20.2}
\leq
N_{0}(d,\delta,p_{j},q_{j}, \rho,\check b^{(0)}_{\rho}(y)) 
\check b^{(0)}_{\rho}(y):=N_{1},
\end{equation}
where $N_{0}(d,\delta,p_{j},q_{j}, \rho, \beta)$ is an increasing
function of $\beta$.
It follows that $M_{T}\leq 
N_{1}+e^{-\lambda \tau}M_{T}$, which implies
\eqref{7.27.3} with $u_{\tau,T}$ in place of
its left-hand side. \qed

\begin{remark}
                          \label{remark 8.1.2}
In the above proof we saw that some parameters
can by made $\cF_{\tau}$-measurable.
This argument is tacitly hidden under our
other arguments.
\end{remark}

The following theorem is a straightforward
consequence of Theorem 1.1.6 of \cite{Kr_26}
and Lemma \ref{lemma 7.27.2}.

\begin{theorem}
                      \label{theorem 7.29.1}
Let
\begin{equation}
                            \label{7.29.2}
\hat p,\hat q\in[1,\infty],\quad
\frac{d}{\hat p}+\frac{1}{\hat q}\leq 1
\end{equation}
and let the assumptions  of Lemma \ref{lemma 7.27.2} be satisfied. Then for any Borel
$f \geq0$ on the set $\{\tau<\infty\}$ we have
\begin{equation}
                              \label{7.29.3}
 E_{\cF_{\tau}}\int_{0}^{\tau_{\rho,y}}e^{-\lambda s}
f(\sft_{s},x_{s})\,dt\leq N \|f\|_{L_{(\hat p,\hat q)}}
\end{equation}
with probability one,
where $N$ depends only on $d,\delta,p_{j},q_{j},
\lambda, \rho,\hat p,\hat q, \check b^{(0)}_{\rho}(y)$.
\end{theorem}
 
Recall that $\lambda_{0}$ and $\nu_{0}$
are introduced in \eqref{10.7.1}.

\begin{lemma}
                    \label{lemma 7.18.1}
 
Assume that $b$ is decomposed as described before
Theorem \ref{theorem 9.19.1}.
 Let $\tau$ be a stopping time.

 We claim that there exists
$\hat b>0$, depending only on $d,\delta,p_{j},q_{j}$,
such that if 
\begin{equation}
                             \label{7.18.010}
 I_{\tau<\infty}\check b^{(\lambda_{0})}_{1}(x_{\tau})\leq \hat b,
\end{equation}
then  
on the set $\{\tau<\infty\}$ we have
\begin{equation}
                          \label{7.6.10}
I:=E_{\cF_{\tau}}e^{-\lambda_{0}\tau_{1}}\leq\nu_{0}
\end{equation}
with probability one,
where $\tau_{1}$ is the first exit time of
$ x_{\tau+s}  $ from $B_{1}(x_{\tau})$.
 
\end{lemma}

Proof. Note that, if $\tau<\infty $,
$$
I\leq E_{\cF_{\tau}}e^{-\lambda_{0}\tau_{1}}I_{\tau_{1}<1}
+e^{-\lambda_{0} }\leq (\cosh 1)^{-1}J+1/6,
$$
where
$$
J:=E_{\cF_{\tau}}e^{-\lambda_{0}(\tau_{1}\wedge1)} p( x_{\tau+ \tau_{1}\wedge1 }-x_{\tau}) ,
$$
and $p( x)=\cosh |x|$.

Simple manipulations
show that
 the function $p (x)$ 
satisfies
$$
D_{r}p= \frac{x^{r}}{|x|}\sinh  |x|,\quad
 D_{rs}p=\frac{x^{r}x^{s}}{|x|^{2}}
\Big( \cosh  |x|-\frac{\sinh  |x|}{ |x|}\Big)
+\delta^{rs} 
\frac{\sinh  |x|}{ |x|}
$$
and, since $\sinh |x|\leq |x| \cosh |x|$, for any symmetric nonnegative $d\times d$-matrix~$a$, 
$$
a^{rs}D_{rs}p-p \tr a
\leq0.
$$
 
In our case $\tr a\leq \lambda_{0} $.
By using It\^o's formula and observing
that $|Dp|\leq   p\leq  \cosh 1$ in $B_{1}$
 and  
$$
|b^{r}(\sft_\tau+s,x_{\tau+s})D_{r} p(x_{\tau+s} - x_{\tau})| 
\leq |b(\sft_{\tau+s},x_{\tau+s})| \cosh 1
$$ 
for $s\leq\tau_{1}$, 
 we see that  
$$
J\leq 1+ (\cosh 1)E _{\cF_{\tau}} \int_{0}^{\tau_{1}\wedge1}
e^{-\lambda_{0} s}|b(\sft_{\tau+s},x_{\tau+s})|\,ds.
$$
Similarly to \eqref{9.20.2} the last
expectation is dominated by
$$
N_{0}(d,\delta,p_{j},q_{j},1, \check b^{(0)}_{1}(x_{\tau})) 
\check b^{(\lambda_{0})}_{1}(x_{\tau})  .
$$
Thus,
$$
I\leq (\cosh1)^{-1} 
\big(1+N\check b^{(\lambda_{0})}_{1}(x_{\tau})   \cosh 1\big) +1/6
$$
$$
\leq (\cosh1)^{-1}+1/6+N\check b^{(\lambda_{0})}_{1}(x_{\tau}) .
$$
Since $(\cosh1)^{-1}<2/3=4/6$, we see that
 there exists $\hat b$ such that \eqref{7.18.010}
implies \eqref{7.6.10}. \qed

\begin{remark}
                        \label{remark 10.7.2}
A very challenging problem is related to this lemma. Does it hold if $\hat b$ is just any number (with $\lambda_{0},\nu_{0}$ depending on this number)? In the time-homogeneous case
this is true with $p=d$.
\end{remark}
 
 Everywhere {\em below in this section}
we impose the following.

\begin{assumption}
                     \label{assumption 7.28.1}
With $n=1,p=p_{1},q=q_{1}$ and $\hat b$ from Lemma \ref{lemma 7.18.1}
we have $\hat b^{(\lambda_{0})}_{1}\leq\hat b$.
 
\end{assumption}

\begin{remark}
                       \label{remark 7.28.1}
As is easy to see, for any $ \rho>0$, $\hat b^{(\lambda_{0})}_{\rho}
\leq N(d, \rho )\hat b$.
\end{remark}

\begin{corollary}
                     \label{corollary 8.30.1}
By iterating \eqref{7.6.10} we get that for
any $\rho>0$, $E_{\cF_{\tau}}e^{-\lambda_{0}\tau_{\rho}}\leq\nu^{\lfloor \rho\rfloor}_{0}$,
where $\tau_{\rho}$ is the first exit time of
$ x_{\tau+s}  $ from $B_{\rho}(x_{\tau})$.
\end{corollary}

\begin{lemma}
                    \label{lemma 7.28.1}
For any 
$\hat p,\hat q $ satisfying
\eqref{7.29.2}, $y\in\bR^{d}$, Borel function $f(t,x)\geq0$, 
such that $f(t,x)=0$ if $x$ is outside $B_{1/2}(y)$,  and stopping time $\tau$,
on $\{\tau<\infty\}$ we have  
\begin{equation}
                          \label{7.28.2}
u_{\tau}:=E_{\cF_{\tau}}\int_{0}^{\infty}e^{-\lambda_{0}s}
f(\sft_{\tau+s},x_{\tau+s}) \,ds\leq N\|f \|_{L_{(\hat p,\hat q)}} 
\end{equation}
with probability one,
where $N$ depends only on $d,\delta,p,q, \hat p,\hat q$.  
\end{lemma}

Proof. We may assume that $f$ is bounded. In this case 
$$
M:=\sup_{\tau}\esssup u_{\tau}<\infty.
$$
 Define
$\tau'$ as the first time $x_{\tau +s}$ hits $\bar B_{1/2}(y) $
 and $\tau''$ as the first time
after $\tau'$ when $x_{\tau+s}$ exits from
$B_{2}(y) $.
Notice that 
by Lemma \ref{lemma 7.18.1},
$$
I_{\tau+\tau'<\infty}
E_{\cF_{\tau+\tau'}}e^{-\lambda_{0}(\tau''-\tau')}\leq \nu_{0}.
$$
Furthermore, owing to Theorem \ref{theorem 7.29.1} 
$$
v_{\tau}:=E_{\cF_{\tau}}\int_{0}^{\tau''}e^{-\lambda_{0}s}
f(\sft_{\tau+s},x_{\tau+s})\,ds
\leq N_{1} \|f\|_{L_{(\hat p,\hat q)}}.
$$

Now observe that  
$$
u_{\tau}=v_{\tau}+E_{\cF_{\tau}}e^{-\lambda_{0}\tau''}
E_{\cF_{\tau+\tau''}}\int_{0}^{\infty}e^{-\lambda_{0}s}
f(\sft_{\tau+\tau''+s},x_{\tau+\tau''+s})
 \,ds
$$
$$
\leq N_{1}\|f\|_{L_{(\hat p,\hat q)}}+\nu_{0}M.
$$
It follows that $M\leq (1-\nu_{0})^{-1}N_{1}
\|f\|_{L_{(\hat p,\hat q)}}$, which proves the lemma.
\qed

\begin{corollary}
                     \label{corollary 7.29.1}
Let a Borel $f \geq0$ be such that
$f(t,x)=0$ for $x\not\in B_{1/2}(y)$ for some $y$ and let $\hat p,\hat q $ satisfy 
\eqref{7.29.2}. Then
\begin{equation}
                            \label{7.29.1}
I:=E_{\cF_{\tau}}\int_{0}^{\infty}e^{-\lambda_{0}s}
f(\sft_{\tau+s},x_{\tau+s})\,ds\leq N 
\Phi(|y-x_{\tau}|)\|f\|_{L_{(\hat p,\hat q)}} 
\end{equation}
with probability one, 
 where $N$ depends only on $d,\delta,p,q, \hat p,\hat q$ and
$$
\Phi(z)=e^{z \ln\nu_{0}}\quad (\ln \nu_{0}<0).
$$

\end{corollary}

Indeed, if $\rho:=|y-x_{\tau}|\leq 2$, then  \eqref{7.29.1} follows from Lemma \ref{lemma 7.28.1}. However, if $\rho>2$, then   for $\gamma$ defined
as the first time $x_{\tau+s}$ hits $\partial
B_{1/2}(y)$, we have
$$
I=E_{\cF_{\tau}}e^{-\lambda_{0}\gamma}
E_{\cF_{\tau+\gamma} }\int_{0}^{\infty}e^{-\lambda_{0}s}
f(\sft_{\tau+\gamma+s},x_{\tau+\gamma+s})\,ds\leq N\|f\|_{L_{(\hat p,\hat q)}}
E_{\cF_{\tau}}e^{-\lambda_{0}\gamma},
$$
where $N$ is from \eqref{7.28.2}.
Then it only remains to notice that
the last expectation is dominated by $\nu_{0}^{2}\Phi(\rho)$ in light of Corollary \ref{corollary 8.30.1}.\qed

\begin{lemma}
                       \label{lemma 7.27.1}
Let
$\hat p,\hat q $ satisfy
\eqref{7.29.2}, $f$ be a  nonnegative Borel function   on $\bR^{d+1}$, and $\tau$
be a stopping time. Then on the set $\{\tau<\infty\}$ we have
\begin{equation}
                          \label{7.27.2}
I:=E_{\cF_{\tau}}\int_{0}^{\infty}e^{-\lambda_{0}s}
f(\sft_{\tau+s},x_{\tau+s})\,ds\leq N\|\Phi^{1/4}(\cdot-x_{\tau})f\|_{L_{(\hat p,\hat q)}} 
\end{equation}
with probability one, 
where $N$ depends only on $d,\delta,p,q, \hat p,\hat q$.
\end{lemma}

Proof. Set $\zeta(x)=cI_{B_{1/2}}(x)$, where
$c^{-1}=|B_{1/2}|$ and for $y\in \bR^{d}$ set
$$
f_{y}(t,x)=f(t,x)\zeta(x-y).
$$
Then by Corollary \ref{corollary 7.29.1}
$$
I=\int_{\bR^{d}}E_{\cF_{\tau}}\int_{0}^{\infty}e^{-\lambda_{0}s}
f_{y}(\sft_{\tau+s},x_{\tau+s})\,dsdy
$$
$$
\leq N\int_{\bR^{d}}\Phi(y-x_{\tau})\|f_{y}\|_{L_{(\hat p,\hat q)}(\bR^{d+1} )}\,dy.
$$

There are a few possible positions for $\hat p,\hat q$.
First let $\infty>\hat p\geq \hat q$. 
Introduce
$$
M_{1}^{\hat q/(\hat q-1)}=\int_{\bR^{d}}
\Phi^{\hat q/(2\hat q-2)}\,dy,\quad \hat q>1,
\quad M_{1}=1, \quad \hat q=1,
$$
$$
M_{2}^{\hat p\hat q/(\hat p-\hat q)}=\int_{\bR^{d}}
\Phi^{\hat p\hat q/(4\hat p-4 \hat q)}\,dy,
\quad \hat p\ne \hat q,\quad M_{2}=1,
\quad \hat p=\hat q.
$$
By observing that $\Phi^{\alpha}...=\Phi^{\alpha/2}(\Phi^{\alpha/2}...)$ and applying 
H\"older's inequality  twice we obtain
$$
I\leq NM_{1}\Big(\int_{\bR} \,dt\int_{\bR^{d}}\Phi^{\hat q/2}(y-x_{\tau})
\Big(\int_{\bR^{d}}f^{\hat p}_{y}(t,x)\,dx\Big)^{\hat q/\hat p}\,dy\Big)^{1/\hat q}
$$
$$
\leq NM_{1}M_{2}
\Big(\int_{\bR} \,dt\Big(\int_{\bR^{d}}\,dx\int_{\bR^{d}}
\,dy\,\Phi^{\hat p/4}(y-x_{\tau})
 f^{\hat p}_{y}(t,x) \Big)^{\hat q/\hat p} \Big)^{1/\hat q}.
$$
Here on account of changing $N$ one can
replace $\Phi^{\hat p/4}(y-x_{\tau})$ with
$\Phi^{\hat p/4}(x-x_{\tau})$ since they are comparable
if $\zeta(x-y)>0$. Then, since
$$
\int_{\bR^{d}}
\,\Phi^{\hat p/4}(x-x_{\tau})
 f^{\hat p}_{y}(t,x)\,dy=N\Phi^{\hat p/4}(x-x_{\tau})
 f^{\hat p} (t,x),
$$
we get \eqref{7.27.2}.

The case that $\infty>\hat q>\hat p$ 
 is quite similar to the previous one.
In case $\hat p\leq \hat q=\infty$
$$
I\leq N\int_{\bR^{d}}\Phi(y-x_{\tau})\Big(
\int_{\bR^{d}}\sup_{t\geq0}|f_{y}(t,x)|^{\hat p}\,dx\Big)^{1/\hat p}\,dy
$$
$$
\leq N\Big(\int_{\bR^{d}}\int_{\bR^{d}}
\Phi^{\hat p/2}(y-x_{\tau})\sup_{t\geq0}|f_{y}(t,x)|^{\hat p}\,dxdy\Big)^{1/\hat p}
$$
$$
\leq N\Big(\int_{\bR^{d}}\int_{\bR^{d}}
 \zeta^{\hat p}(x-y)\sup_{t }|\Phi^{1/2}(x-x_{\tau})f (t,x)|^{\hat p}\,dxdy\Big)^{1/\hat p}
$$
$$
=N\|\Phi^{1/2}(\cdot-x_{\tau})f\|_{L_{(\hat p,\hat q)}}.
$$

In the remaining  case   $\hat p=\infty>\hat q$  we have
$$
I\leq N\int_{\bR}\Phi(y-x_{\tau})\Big(\int_{\bR}\sup_{x\in\bR^{d}}|f_{y}(t,x)|^{\hat q}\,dt\Big)^{1/\hat q}\,dy
$$
$$
\leq N\Big(\int_{\bR}\int_{\bR^{d}}\Phi^{\hat q/4}(y-x_{\tau})
\sup_{x\in\bR^{d}}|\Phi^{1/4}(x-x_{\tau})f (t,x)|^{\hat q}\,dydt\Big)^{1/\hat q}
$$
$$
=N\|\Phi^{1/4}(\cdot-x_{\tau})f\|_{L_{(\hat p,\hat q)}}.
$$

 \qed

\begin{theorem}
                       \label{theorem 7.31.1}
Under the assumption of Lemma \ref{lemma 7.27.1}  on the set $\{\tau<\infty\}$ we have
 (notice $2\lambda_{0}$)
\begin{equation}
                          \label{7.27.20}
 I:=E_{\cF_{\tau}}\int_{0}^{\infty}e^{-2\lambda_{0} s}
f(\sft_{\tau+s},x_{\tau+s})\,ds\leq N_{0}\sup_{C\in\bC_{1}}
\| \Psi(\cdot-\sft_{\tau},\cdot-x_{\tau})f \|_{L_{(\hat p,\hat q)}(C)} 
\end{equation}
with probability one,
where $N_{0}=N_{0}( d,\delta,p,q, \hat p,\hat q) 
  $ and 
  $$
  \Psi(t,x)=e^{-\lambda_{0}t/2}\Phi^{1/8}(x)I_{t>0}.
  $$
Furthermore, if
$$
M:=\sup_{C\in\bC_{1}}
\|  f \|_{L_{(\hat p,\hat q)}(C)}<\infty,  
$$
then
\begin{equation}
                          \label{8.2.1}
  E_{\cF_{\tau}}\exp\Big(\frac{1}{2MN_{0}}\int_{0}^{\infty}e^{-2\lambda_{0} s}
f(\sft_{\tau+s},x_{\tau+s})\,ds\Big)\leq 2
\end{equation}
with probability one.
\end{theorem}

Proof. Note that
$$
e^{-2\lambda_{0} s}
f(\sft_{\tau+s},x_{\tau+s})=e^{-\lambda_{0} s}
\big(e^{-\lambda_{0}\sft_{\tau+ s}}
f(\sft_{\tau+s},x_{\tau+s})I_{\sft_{\tau+s}
\geq \sft_{\tau}}e^{\lambda \sft_{\tau}}\big).
$$
Then by Lemma \ref{lemma 7.27.1}  
$$
I\leq N \|fe^{\lambda_{0}(\sft_{\tau}-\cdot)}\Phi^{1/4}(\cdot-
x_{\tau})\|
_{L_{(\hat p,\hat q)}(\bR^{d+1}_{\sft_{\tau}} )},
$$
where $\bR^{d+1}_{t}=[t,\infty)\times\bR^{d}$.
 
Next, let $\zeta=c^{-1}I_{C_{1}}$, where $c=|C_{1}|$
and $g(t,x)=f(t,x)e^{ \lambda_{0}(\sft_{\tau}-t)}\Phi^{1/4}(x-x_{\tau})$.
Then for $t\geq \sft_{\tau}$
$$
g(t,x)=\int_{\bR^{d+1}_{\sft_{\tau}}}\zeta(s-t,y-x)g(t,x)
\,dyds
$$
$$
\leq N\int_{\bR^{d+1}_{\sft_{\tau}}}\big(e^{\lambda_{0}(\sft_{\tau}-s)/2}\Phi^{1/8}(y
-x_{\tau})\big)\zeta(s-t,y-x)\Psi (t-\sft_{\tau},x-x_{\tau})f(t,x)\,dyds.
$$
It follows by Minkowski's inequality that
$$
I\leq N\int_{\bR^{d+1}_{\sft_{\tau}}}
e^{ \lambda_{0}(\sft_{\tau}-s)/2}\Phi^{1/8}(y-x_{\tau})
\|\zeta(s-\cdot,y-\cdot) \Psi(\cdot-\sft_{\tau},\cdot-x_{\tau})f\|
_{L_{(\hat p,\hat q)} }\,dyds.
$$
Now, to get \eqref{7.27.20}, it only remains to observe that the last
norm is less than $1/c$ times $\|  \Psi(\cdot-\sft_{\tau},\cdot-x_{\tau})f\|
_{L_{(\hat p,\hat q)}(C_{1}(s-1,y)) }$. 

Since the right-hand side of \eqref{7.27.20}
is less than $N_{0}M$ and $\tau$ is arbitrary,
\eqref{8.2.1} follows from \eqref{7.27.20}
by Khasminskii's lemma.
\qed

\mysection{Solvability of \eqref{4.27.10}}

We  need the
  following two
  results  due to A.V. Skorokhod
(see Ch.~1, \S6 and Ch.~2, \S3 in \cite{Sk_61}).  
\begin{lemma}
                           \label{lemma 4.23.1}
  Suppose that $d_1$-dimensional random
 processes $\xi^{(n)}_{t}$  $(t\geq 0, n =  
1,2, . . .)$ are defined
 on some probability spaces equipped with probability measures $P^{n}$. Assume that for each $T> 0$ and
$\varepsilon > 0$
\begin{equation}
                                                  \label{5.10.5}
\lim_{c\to\infty} \sup_{n} \sup_{t\leq T} P^{n} (|\xi^{(n)}_{t}|>c) = 0,
\end{equation}
\begin{equation}
                                                  \label{5.10.6}
\lim_{h\downarrow 0} \sup_{n} \sup_{\substack{t_{1},t_{2}\leq T\\
|t_{1}-t_{2}|\leq h}} P^{n}(|\xi^{(n)}_{t_{1}}-\xi^{(n)}_{t_{2}}|>
\varepsilon)=0.
\end{equation}
Then  one can find a sequence of integers $n'\to\infty$,
 a probability space equipped with a probability measure
 $P$, and
random processes $\tilde \xi_{t},\tilde \xi_{t}^{(n')}$
 defined on this probability space such that all finite-dimensional
distributions of $\tilde \xi_{t}^{(n')}$ coincide with 
the corresponding finite-dimensional
distributions of $\xi_{t}^{(n')}$ and 
$$
P (|\tilde \xi_{t}-\tilde \xi_{t}^{(n')}|>\varepsilon) \to 0   
$$
as $n'\to\infty$ for any $\varepsilon>0$ and $t\geq0$.
\end{lemma}

\begin{lemma}
                                               \label{lemma 4.23.2}
Suppose that on a complete
 probability space $(\Omega,\cF,P)$
we are given random processes $\xi^{(n)}_{t}$,
$w^{(n)}_{t}$, $n=0,1,2,...$. Suppose that
the assumptions of Lemma  \ref{lemma 4.23.1} 
 are satisfied and
\begin{equation}
                                                     \label{5.14.1}
 \xi^{(n)}_{t}\to \xi^{(0)}_{t},\quad w^{(n)}_{t}\to w^{(0)}_{t}
\end{equation}
in probability as $n\to\infty$
for each $t\geq0$. Finally, assume that   $w^{(n)}_{t}$ are $d_{1}$-dimensional
Wiener processes relative to some
increasing families of complete $\sigma$-fields
$\cF_{t}^{n}\subset\cF$,
$t\geq0$, $n=0,1,2,...$, the functions 
$\xi^{(n)}_{t}(\omega)$ are bounded on $[0,\infty)\times\Omega$
uniformly in $n$, and each of them is
progressively measurable relative to $\cF^{n}_{t}$. 

Then the stochastic integrals 
$$
I^{n}_{t}:=\int_{0}^{t}\xi^{(n)}_{s}\,dw^{(n)}_{s}
$$
 are well defined for $t\geq0$, $n=0,1,2,...$ and
   $I^{n}_{t}\to I^{0}_{t}$
in probability as $n\to\infty$
for each $t\geq0$.
\end{lemma}

\begin{remark}
                                                 \label{remark 5.11.1}
As it follows from the proof of Lemma \ref{lemma 4.23.2}
given in \cite{Sk_61}
  we need conditions
\eqref{5.10.5}, \eqref{5.10.6}, and \eqref{5.14.1} to hold
only for $t,t_{1},t_{2}$ restricted to a set
of full measure in order for the assertion of the lemma to be true.
\end{remark}

 In the following Lemma 3.1.4 of \cite{Kr_25} the function $\sigma_{0}(t,x)$
is a bounded Borel $d\times d_{1}$-matrix valued
function on $\bR^{d+1}_{0}=[0,\infty)\times\bR^{d}$,
$b_{0}(t,x)$ is a Borel $\bR^{d}$-valued function
defined on the same set.

\begin{lemma} 
                                             \label{lemma 5.12.1}
Let $\bR^{d+d_{1}}$-valued processes $(x'_{t},w'_{t})$,  $(x''_{t},w''_{t})$ $t\geq0$,
$i=1,2$, 
defined on perhaps different complete probability spaces, have the same
finite-dimensional distributions. Define 
$\cF'_{t}$ as the  completion of
  $\sigma (x'_{s}, w'_{s}:s\leq t)$
and similarly define $\cF''_{t}$, assume that $w'_{t}$
is a Wiener process with respect to   $\cF'_{t}$ and $w''_{t}$
is a Wiener process with respect to   $\cF''_{t}$. Also suppose that $x'_{t}$ is continuous and
(a.s.) for all $t\geq 0$
\begin{equation}
                                              \label{5.13.1}
\int_{0}^{t}|b_{0}(s,x'_{s})|\,ds<\infty,\quad
x'_{t}=\int_{0}^{t} \sigma_{0}(s,x'_{s})\,dw'_{s}+
\int_{0}^{t}b_{0}(s,x'_{s})\,ds.
\end{equation}
Then $x''_{t},w''_{t}$ have   modifications 
(called again $x''_{t},w''_{t}$) such that $w''_{t}$
is a Wiener process with respect to   $\cF''_{t}$
and (a.s.) for all $t\geq 0$
\begin{equation}
                                              \label{5.13.2}
\int_{0}^{t}|b_{0}(s,x''_{s})|\,ds<\infty,\quad
x''_{t}=\int_{0}^{t} \sigma_{0}(s,x''_{s})\,dw''_{s}+
\int_{0}^{t}b_{0}(s,x''_{s})\,ds.
\end{equation}
\end{lemma}

To move forward we need one more
assumption.

\begin{assumption}
                      \label{assumption 8.16.1}

We are given $(t^{(n)}_{0},x^{(n)}_{0})\in\bR^{d+1} $, $n= 1,2,...$,   such that 
$$
(t^{(n)}_{(0)},x^{(n)}_{(0)})
\to (t^{(0)}_{(0)},x^{(0)}_{(0)}):=(t_{(0)},x_{(0)})
$$  as $n\to\infty$. 
We are also given   $b^{(n)}(t,x)$, $n= 1,2,...$,   $\bR^{d}$-valued 
Borel functions
on $\bR^{d+1} $, satisfying 
Assumption \ref{assumption 7.28.1}
and such that, for each $n$, there exists a   complete probability space $(\Omega^{n},\cF^{n}, P^{n})$,  
an increasing filtration of  complete $\sigma$-fields $\cF^{n}_{t}\subset \cF^{n}$, $t\geq0$,
a process $w^{(n)}_{t}$, $t\geq0$, which is a $d$-dimensional Wiener process
relative to $\{\cF^{n}_{t}\}$, and an $\cF^{n}_{t}$-adapted continuous
process $x^{(n)}_{t}$ such that $P^{n}$-(a.s.) for all   $t\geq0$
\begin{equation}
                                                 \label{11.29.1}
x^{(n)}_{t}=x^{(n)}_{(0)} +\int_{0}^{t}\sigma (t^{n}_{(0)}+s,x^{(n)}_{s}) \,dw^{(n)}_{s}
+\int_{0}^{t}b^{(n)}(t^{n}_{(0)}+s,x^{(n)}_{s}) \,ds.
\end{equation} 

\end{assumption}

The statement of the following result coincides
with Theorem 1.6.1 of \cite{Kr_26} apart from
the last assertion in the latter that
the distributions of $x^{(n)}_{\cdot}$ are tight provided $b\in L_{(p,q)}$ and $p\geq q$.

\begin{theorem}
              \label{theorem 9.6.4} 
(i) There exists 
a probability space $(\Omega ,\cF ,P )$,
a filtration of $\sigma$-fields $\cF _{t}\subset \cF $, $t\geq0$,
a process $w _{t}$, $t\geq0$, which is a $d$-dimensional Wiener process
relative to $\{\cF _{t}\}$, and an $\cF _{t}$-adapted
process $x_{t}$ such that 
 (a.s.) for all   $t\geq0$ equation \eqref{4.27.10} holds;

(ii) under Assumption \ref{assumption 8.16.1} additionally suppose that  $b^{(n)}\to b^{(0)}:=b $ 
as $n\to\infty$ in  $L_{(p,q)}(C)$ for any $C\in\bC$.
Then  a subsequence of  finite dimensional distributions of 
$ x^{(n)}_{\cdot} $ converges weakly to 
the corresponding distributions of one of the solutions of   
\eqref{4.27.10}  described in (i).  
\end{theorem}

Proof. First we prove assertion (ii).

{\em Step 1\/}.  
For $M>0$ define 
$$
\xi^{(n)}_{t}=\int_{0}^{t}b^{(n)}(t^{(n)}_{0}+s,x^{(n)}_{s})\,ds,
$$
$$
\xi^{(n)M}_{t}=\int_{0}^{t}b^{(n)}(t^{(n)}_{0}+s,x^{(n)}_{s})
I_{|b^{(n)}( t^{(n)}_{0}+s, x^{(n)}_{s})|\leq M}\,ds.
$$

Since the derivative of $\xi^{(n)M}_{t}$ 
with respect to $t$ is bounded,
both conditions \eqref{5.10.5} and \eqref{5.10.6}
are satisfied for $\xi^{(n)M}_{t}$. Furthermore,
for any $T,R\in(0,\infty)$ by Theorem
\ref{theorem 7.31.1}
$$
E^{n}I_{\sup_{s\leq T}|x^{(n)}_{s}|<R}\int_{0}^{T}|b^{(n)}(t^{(n)}_{0}+s,x^{(n)}_{s})|
I_{|b^{(n)}(t^{(n)}+s,x^{(n)}_{s})|\geq M}\,ds
$$
$$
\leq E^{n} \int_{0}^{T}|b^{(n)}(t^{(n)}_{0}+s,x^{(n)}_{s})|
I_{|b^{(n)}(t^{(n)}_{0}+s,x^{(n)}_{s})|\geq M,|x^{(n)}_{s}|<R}\,ds 
$$
\begin{equation}
                                  \label{8.14.3}
\leq  
N\|b^{(n)}I_{|b^{(n)}|\geq M}\|_{L_{(q,p)}
(C_{2S,R}(-S,0))},
\end{equation}
where $N$ is independent of $n$ and $S<\infty$
is such that $-S\leq t^{(n)}_{0}+T\leq S$ for all $n$. Since $b^{n}\to b$
in the $\|\cdot\|_{L_{(q,p)}
(C_{2S,R}(-S,0))}$-norm, the latter quantity can be made
as small as we like on account of choosing $M$
 large enough.  Hence for any  
  $c,M,R\in[2,\infty)$
$$
 P
(|\xi_{t}^{(n)}|\geq c)\leq
 P
\big(|\xi_{t}^{(n)M}|\geq c-1\big)+\varepsilon(M,R)
+P(
\sup_{s\leq T}|x^{(n)}_{s}|\geq R),
$$
where $\varepsilon(M,R)\to0$ as $M\to\infty$
for each $R$. Taking into account that
\begin{equation}
                                \label{8.16.6}
P(\sup_{s\leq T}|x^{(n)}_{s}|\geq R)\leq e^{\lambda_{0} T}\nu^{R-1}_{0}
\end{equation}
  and recalling what is said
about $\xi_{t}^{(n)M}$ we see without much
difficulty that condition \eqref{5.10.5}
is satisfied. Similar arguments show  that
\eqref{5.10.6} is satisfied as well.  

The fact that \eqref{5.10.5} and \eqref{5.10.6}
are satisfied for
$$
\eta^{(n)}_{t}=x^{(n)}_{0}+\int_{0}^{t}\sigma (t^{(n)}_{0}+s,x^{(n)}_{s}) \,dw^{(n)}_{s}
$$
us almost trivial.

{\em Step 2\/}. By Lemma \ref{lemma 4.23.1}  there is a subsequence, which by common abuse of notation
we identify with the original one, a probability space, and
random $\bR^{2d}$-valued
processes $(\tilde x^{(n)}_{t},\tilde w^{(n)}_{t})$,
$(\tilde x^{(0)}_{t},\tilde w^{(0)}_{t})$  
defined on this probability space such that all finite-dimensional
distributions of $(\tilde x^{(n)}_{t},\tilde w^{(n)}_{t})$ coincide with 
the corresponding finite-dimensional
distributions of $(x_{t}^{(n)},w^{(n)}_{t})$ and 
\begin{equation}
                                                       \label{5.13.6}
P (|(\tilde x^{(n)}_{t},\tilde w^{(n)}_{t})
-(\tilde x^{(0)}_{t},\tilde w^{(0)}_{t})|\geq \varepsilon) \to 0   
\end{equation}
as $n \to\infty$ for any $\varepsilon>0$ and $t\geq0$.
 
Furthermore (as a result of   
\eqref{5.10.5}), for any $T\in(0,\infty)$ as $R\to
\infty$
\begin{equation}
                           \label{8.14.1}
P (|\tilde x^{(n)}_{t}|>R)\to0
\end{equation}
uniformly with respect to $t\leq T$ and $n\geq 1$ and, 
as \eqref{5.13.6}   implies, with respect to
$n\geq 0$ as well.

For $n\geq0$ introduce $\tilde \cF^{n}_{t}$ as the completion
of $\sigma(\tilde x^{(n)}_{s}, \tilde w^{(n)}_{s},s\leq t)$. It is easy to see,
using Kolmogorov's continuity criterion, that
$\tilde w^{(0)}_{t}$ admits a continuous modification
$\hat w^{(0)}_{t}$ such that
$\{\hat w^{(0)}_{t},\tilde \cF^{0}_{t}\}$  is a Wiener
process.

By Lemma \ref{lemma 5.12.1}, for each $n\geq1$, the process
$(\tilde x^{(n)}_{t},\tilde w_{t}^{(n)})$
admits a continuous
modification denoted by   $(\hat x^{(n)}_{t},\hat w_{t}^{(n)})$
such  that 
 $(\hat w_{t}^{(n)},
\tilde \cF_{t}^{n})$ is a Wiener process and  (a.s) for all $t\geq0$
\begin{equation}
                                                     \label{5.12.1}
\hat x^{(n)}_{t}=x^{(n)}_{0}+\int_{0}^{t}\sigma (t^{(n)}_{0}+s,
\hat x^{(n)}_{s})\,d\hat w^{(n)}_{s}+\int_{0}^{t}b^{(n)}(t^{(n)}_{0}+s,
\hat x^{(n)}_{s})\,ds.
\end{equation}

In light of \eqref{5.13.6}   we have
\begin{equation}
                                                       \label{5.17.1}
P (|(\hat x^{(n)}_{t},\hat w^{(n)}_{t})
-(\tilde x^{(0)}_{t},\tilde w^{(0)}_{t})|\geq \varepsilon) \to 0   
\end{equation}
as $n \to\infty$ for each $\varepsilon>0$ and $t\geq0$.

Now the fact that $\tilde x^{(0)}_{t}$ may be not measurable
in $t$ causes some problems. However, 
set $\phi(x)=x/(1+|x|)$  and observe that, owing to \eqref{5.17.1},
 $\phi(\hat x^{(n)}_{t})$ form a Cauchy sequence
in $L_{1}(\Omega\times[0,T])$ and, hence, converges in that space
to $\phi(\hat x^{(0)}_{t})$, where $\hat x^{(0)}_{t}$ is measurable
with respect to $(\omega,t)$. By Fubini's theorem
there is a set $\cS\subset [0,\infty)$ of full measure
such that, for any $t\in\cS$, $\hat x^{(0)}_{t}=\tilde x^{(0)}_{t}$
(a.s.).   We set
$\hat x^{(0)}_{t}=0$ for $t\not\in \cS$ and 
observe that $\hat x^{(0)}_{t}$
is $\tilde \cF^{0}_{t}$-adapted.

Also note that   \eqref{5.17.1}  remains 
valid if we replace $(\tilde x^{(0)}_{t},\tilde w^{(0)}_{t})$
with $(\hat x^{(0)}_{t},\hat w^{(0)}_{t})$ and 
restrict the range  of $t $ to $t \in\cS$.
This is done to accommodate Remark \ref{remark 5.11.1}. Observe in addition that \eqref{8.14.1} implies that ($P (|\hat x^{(0)}_{t}|>R)$ is a measurable function)
\begin{equation}
                         \label{8.16.7}
\int_{0}^{T}P (|\hat x^{(0)}_{t}|>R)\,dt\to0
\end{equation}
as $R\to\infty$ for any $T>0$.

{\em Step 3\/}.
 Now by Lemma
\ref{lemma 4.23.2} for any  $T\geq0$ and 
bounded continuous
$d\times d$ symmetric matrix-valued
 $\alpha(t,x)$ we have 
\begin{equation}
                                                     \label{5.13.7}
 \int_{0}^{T}\alpha(t^{(n)}_{0}+s,\hat x_{s}^{(n)})\,d\hat w^{(n)}_{s}
\to \int_{0}^{T}\alpha(t^{(0)}_{0}+s,\hat x_{s}^{(0)})\,d\hat w^{(0)}_{s}
\end{equation}
as $n\to\infty$ in probability.
We want to use this  to pass to the limit in the stochastic term
in \eqref{5.12.1}.
But first observe that
by Lemma \ref{lemma 7.27.1}  
for any $T\in(0,\infty)$, Borel $f(t,x)\geq0$, and $n\geq1$  
\begin{equation}
                                                  \label{4.26.1}
E\int_{0}^{T}
f(t,\hat x^{n}_{t})\,dt\leq N\|\bar\Phi fI_{(0,T)}\|_{L_{(p,q)}},
\end{equation} 
where $N$ is independent of $f$ and $n$,
$$
\bar\Phi(x)=\sup_{n}\Phi^{1/4}(x-x^{(n)}_{0}).
$$
 The convergence
in probability implies that \eqref{4.26.1} holds for
$n=0$ as well with the same constant $N$, first for 
nonnegative $f\in C^{\infty}_{0}
(\bR^{d+1})$ and then, due to general measure-theoretic arguments,
for any Borel nonnegative $f$.

Then take an $\alpha$ as above with values in $\bS_{\delta}$ and write
$$
P\Big(\Big|\int_{0}^{T}\sigma(t^{(n)}_{0}+s,\hat x_{s}^{(n)})\,d\hat w^{(n)}_{s}
- \int_{0}^{T}\sigma(t^{(0)}_{0}+s,\hat x_{s}^{(0)})\,d\hat w^{(0)}_{s}\Big|\geq 3\varepsilon
\Big)\leq I_{n}+I_{0}+J_{n},
$$
where
$$
I_{n}=P\Big(\Big|\int_{0}^{T}(\sigma-\alpha)(t^{(n)}_{0}+s,\hat x_{s}^{(n)})\,d\hat w^{(n)}_{s}
  \Big|\geq \varepsilon
\Big)
$$
$$
J_{n} =P\Big(\Big|\int_{0}^{T}\alpha(t^{(n)}_{0}+s,\hat x_{s}^{(n)})\,d\hat w^{(n)}_{s}
- \int_{0}^{T}\alpha(t^{(0)}_{0}+s,\hat x_{s}^{(0)})\,d\hat w^{(0)}_{s}\Big|\geq \varepsilon
\Big).
$$
In light of \eqref{5.13.7}, $J_{n}\to0$
as $n\to\infty$. Estimate \eqref{4.26.1}
implies that
$$
I_{n}\leq N\varepsilon^{-1}
\|\bar\Phi (\sigma-\alpha)I_{(-S,S)}\|_{L_{(p,q)}},
$$
which can be made as small as we wish by choosing  appropriate $\alpha$. It follows that
\begin{equation}
                               \label{8.3.1}
\int_{0}^{t}\sigma (t^{(n)}_{0}+s,
\hat x^{(n)}_{s})\,d\hat w^{(n)}_{s}
\to \int_{0}^{t}\sigma (t^{(0)}_{0}+s,
\hat x^{(0)}_{s})\,d\hat w^{(0)}_{s}
\end{equation}
in probability as $n\to\infty$.
 
To deal with the second integral in \eqref{5.12.1},
take a bounded continuous $\bR^{d}$-valued
$\beta$ on $\bR^{d+1}$ and observe that,
thanks to \eqref{4.26.1},   
$$
E\Big|\int_{0}^{T}\big(b^{(n)}(t^{(n)}_{0}+s,
\hat x^{(n)}_{s}) -\beta(t^{(n)}_{0}+s,
\hat x^{(n)}_{s})\big)\,ds\Big|
$$
$$
\leq N\|\bar\Phi (b^{(n)}-\beta)I_{(-S,S)}\|_{L_{(p,q)}}\leq N\|\bar\Phi (b^{(n)}-b^{(0)})I_{(-S,S)}\|_{L_{(p,q)}}
$$
$$
+N\|\bar\Phi (b^{(0)}-\beta)I_{(-S,S)}\|_{L_{(p,q)}}.
$$
Here the first term in the second inequality
tends to zero as $n\to\infty$ by assumption
and the second term can be made as small as we like on account of choosing appropriate $\beta$.
In addition,
$$
\int_{0}^{T}E|\beta(t^{(n)}_{0}+s,
\hat x^{(n)}_{s}) -\beta(t^{(0)}_{0}+s,
\hat x^{(0)}_{s})\big|\,ds\to0
$$
as $n\to\infty$ because $\beta$ is bounded
and continuous and $\hat x^{(n)}_{s}\to
\hat x^{(0)}_{s}$ in probability.
  
Hence,   in probability
$$
\int_{0}^{t}b^{(n)}(t_{n}+s,
\hat x^{(n)}_{s})\,ds\to \int_{0}^{t}b^{(0)}(t_{0}+s,
\hat x^{(0)}_{s})\,ds.
$$

This and \eqref{8.3.1} allow us to pass to the limit
in \eqref{5.12.1} when $t\in \cS$ (and 
$\hat x^{(n)}_{t}\to \hat x^{(0)}_{t}$ in probability) and shows that
\eqref{5.12.1} holds true for $n=0$ and
$t\in \cS$. In turn this implies that
$\hat x^{(n)}_{t}$ is extendible from the 
set of full measure $\cS$ to all $t$ (as the right-hand side of \eqref{5.12.1}) as a continuous
function satisfying \eqref{5.12.1} with $n=0$
for all $t$ at once (a.s.).
This proves assertion (ii) of the theorem.

Assertion (i) follows from assertion (ii)
after we take $t^{(n)}_{0}=t_{0},x^{(n)}_{0}=x_{0}$ and $b^{(n)}(t,x)=\zeta(nt,nx)b(t,x)$,  
where $\zeta\in C^{\infty}_{0}(\bR^{d+1})$
is such that $\zeta(0)=1$. Indeed, in this
case $b^{(n)}\in L_{(p,q)}$ and by Theorem
\ref{theorem 6.19.10} and Remark \ref{remark 8.5.1}
the assumptions stated before Theorem
\ref{theorem 9.6.4} are satisfied.
This proves the theorem. \qed

\mysection{Proof of Theorem \ref{theorem 6.19.1}}
                    \label{section 9.12.1}

In case $\rho_{0}=1$ in \eqref{7.18.1} our theorem
is proved by repeating word for word the
corresponding (rather long) part of the
proof of Theorem 1.9.1 of \cite{Kr_26} 
(originated in \cite{Kr_73_1}) using
Theorem \ref{theorem 9.6.4} in place of Theorem 1.6.1 
of \cite{Kr_26}.  
 
For $\rho_{0}\ne1$ we use the self-similar transformation. 
Namely, introduce 
$$
a _{0}(t,x)=a(\rho_{0}^{2}t, \rho_{0}x),\quad
 b_{0} (t,x)=\rho_{0} b(\rho_{0}^{2}t, \rho_{0}x).
$$
It is not hard to see that
$$
\widehat{(b_{0})}^{(\lambda_{0})}_{ 1}=\rho_{0} 
\hat b_{R}^{(\lambda_{0}/\rho_{0}^{2})}\leq \hat b.
$$
Hence, there exists a strong Markov diffusion
process $X^{0}=\big((\sft_{t},x_{t}),\cN_{t},
P^{0}_{t,x}\big)$ corresponding to $a_{0},b_{0}$
 as described in the Introduction.

In particular, as pointed out in Remark
\ref{remark 8.5.1}, for any $(t,x)\in\bR^{d+1}$
the system  
$$
x_{s}=x+\int_{0}^{s}\sigma(\rho_{0}^{2}\sft_{r},\rho_{0}x
_{r})\,dw_{r}+\int_{0}^{s}\rho_{0}b(\rho_{0}^{2}\sft_{r},\rho_{0}x
_{r})\,dr,\quad \sft_{s}=t+s
$$
has a solution whose distribution is given by
$P^{0}_{t,x}$. Then the couple
$$
x^{0}_{t}:=\rho_{0}x_{\rho_{0}^{-2}t},\quad\sft^{0}_{t}=
\rho_{0}^{2}\sft_{\rho_{0}^{-2}t}
$$
satisfies
$$
x^{0}_{s}=\rho_{0}x+\int_{0}^{s}\sigma(\sft^{0}_{r},
x^{0}_{r})\,dw^{0}_{r}
+\int_{0}^{s}b(\sft^{0}_{r},
x^{0}_{r})\,dr,\quad \sft^{0}_{s}=\rho_{0}^{2}t+s,
$$
where $w^{0}_{t}=\rho_{0}w_{\rho_{0}^{-2}t}$ is a Wiener process.
 
Now define new measures $P_{t,x}$ on $(\Omega,\cN_{\infty})$ by the formula
$$
P_{\rho_{0}^{2}t,\rho_{0}x}(A)=P^{0}_{t,x}\big(({\sft^{0}_{\cdot}}
,x^{0}_{\cdot})\in A\big).
$$
We leave to the reader the routine checking that so described changes of the scales in $\bR^{d+1}$
yields a strong Markov diffusion process
$\big((\sft_{t},x_{t}),\cN_{t},
P_{t,x}\big)$ corresponding to the original
$a,b$. \qed

\begin{remark}
                      \label{remark 8.13.03}
For the process from Theorem \ref{theorem 6.19.1},
$$\lambda_{0}\geq  
(\delta^{-1}d)\vee \ln6 ,\quad\nu_{0}=11/12
$$
  we have
\begin{equation}
                          \label{7.6.01}
 E_{0,0}e^{-\lambda_{0}\rho_{0}^{-2}\tau_{\rho_{0}}}\leq\nu_{0},
\end{equation}
where $\tau_{\rho_{0}}$ is the first exit time of
$ x _{ s}  $ from $B_{\rho_{0}} $.
Indeed, the $P_{0,0}$ distribution 
of $\rho_{0}^{-2}\tau_{\rho_{0}}$ is given by the
$P^{0}_{0,0}$ distribution 
of $ \tau_{1}$, so that the result follows from
Lemma \ref{lemma 7.18.1}.

\end{remark}

\mysection{Some properties of solutions
of \eqref{4.27.10} under the condition
of Theorem \ref{theorem 6.19.01}}
                     \label{section 8.16.1}

In the setting of Section \ref{section 8.2.1}
  we suppose that \eqref{8.13.1} 
holds with $\hat b$ from Lemma \ref{lemma 7.18.1}
when $n=1$, $p_{1}=p,q_{1}=q$.  Then
define recursively
$$
R_{0}=0,\quad\rho_{0}=\rho(R_{0}),\quad ,
\quad R_{1}=  \rho_{0},\quad\rho_{1}=
\rho(R_{1}),\quad R_{2}=\rho_{0}+\rho_{1},
$$
$$
 R_{n}=\rho_{0}+...+
\rho_{n-1},\quad \rho_{n}=\rho(R_{n}),
$$
so that $R_{n+1}-R_{n}=\rho(R_{n})$.
Obviously, $R_{n}$ is increasing and $R_{n_{0}}
>e$ for some $n_{0}$. 

\begin{lemma}
                         \label{lemma 8.16.1}
  We have
\begin{equation}
                             \label{8.15.1}
\bar R:=\lim_{n\to\infty}R_{n}=\infty,
\end{equation}
and for $m=0,1,...$, $n_{1}\geq n_{0}$
\begin{equation}
                             \label{8.15.2}
 \mu(n_{1},m):=\rho^{2}_{n_{1}}+...+
\rho^{2}_{n_{1}+m}\geq K_{0}\ln\ln R_{n_{1}+m+1} -K_{0}\ln\ln R_{n_{1}}.
\end{equation}

\end{lemma}

Proof. For $n\geq n_{0}$ we have $\rho(R_{n})\geq K_{0}R_{n}^{-1}(\ln R_{n})^{-1}$,
$$
R_{n+1}^{2}=R_{n}^{2}+2R_{n}\rho(R_{n})
+\rho^{2}(R_{n})\geq R_{n}^{2}+2K_{0} (\ln R_{n})^{-1} ,
$$
$\bar R\geq \bar R+2K_{0}(\ln \bar R)^{-1}$, 
and \eqref{8.15.1} follows.
To prove \eqref{8.15.2}, it suffices to observe that for $n\geq n_{0}$ we have 
$$
\rho_{n}^{2}=\rho_{n}\rho_{n}\geq K_{0}(\ln R_{n})^{-1}\Big(\frac{R_{n+1}}{R_{n}}-1
\Big)\geq K_{0}(\ln R_{n})^{-1}\ln\Big (\frac{R_{n+1}}{R_{n}} 
\Big)
$$
$$
=K_{0}\Big(\frac{\ln R_{n+1}}{\ln R_{n}}-1\Big)
\geq K_{0}\ln \Big(\frac{\ln R_{n+1}}{\ln R_{n}}\Big).
$$
  \qed

The fact that $R_{n}\to\infty$ as $n\to
\infty$ and $\dashnorm b\|_{L_{(p,q)}(C_{  
\hat\rho}(t,x))}<\infty$ leads to the following.
 \begin{corollary}
                     \label{corollary 8.16.1}
It holds that for any $\rho>0$ and $C\in \bC_{\rho}$,
$\dashnorm b\|_{L_{(p,q)}(C)}<\infty$.
\end{corollary}

\begin{lemma}
                        \label{lemma 8.13.2}
For  each  stopping time $\tau$, $\rho=\hat\rho(x_{\tau})$,
and
 $$
\lambda_{0}\geq
(\delta^{-1}d)\vee \ln6 ,\quad\nu_{0}=11/12,
$$
on the set $\{\tau<\infty \}$ we have 
\begin{equation}
                          \label{7.6.1}
 E_{\cF_{\tau}}e^{-\lambda_{0}\rho^{-2}\tau_{\rho}}\leq\nu_{0},\quad
E_{\cF_{\tau}}e^{-\lambda_{0}\tau_{\rho}}\leq\nu^{\rho^{2}}_{0},
\end{equation}
where $\tau_{\rho}$ is the first exit time of
$ x_{\tau+ s}  $ from $B_{\rho}(x_{\tau} )$.

\end{lemma}

The first estimate is deduced from 
Lemma \ref{lemma 7.18.1} by using the self-similarity as in Remark \ref{remark 8.13.03}
and the second one follows from  H\"older's
inequality.

Next,  let $\tau^{n}$ be the first time
$x_{s}$ exits from $B_{R_{n+1}}$, $n\geq1$. 
\begin{lemma}
                        \label{lemma 8.16.3}
For $m=0,1,...$, $n_{1}\geq n_{0}$,
$|x_{0}|\leq R_{n_{1}}$ it holds that
\begin{equation}
                             \label{8.17.3}
P (\max_{s\leq T}|x_{s}|\geq R_{n_{1}+m+1})
\leq e^{ \lambda_{0}T}
\nu^{\mu(n_{1},m)}_{0}.
\end{equation}
\end{lemma}

Proof.
Owing to 
the second inequality in \eqref{7.6.1}  we get
$$
E e^{-\lambda_{0}\tau^{n_{1}+m}}\leq \nu^{\mu(n_{1},m)}_{0}.
$$
After that it only remains to use that
$$
P (\max_{s\leq T}|x_{s}|\geq R_{n_{1}+m+1})= P (\tau_{n_{1}+m}\leq T)
\leq e^{ \lambda_{0}T}
Ee^{-\lambda_{0}\tau^{n_{1}+m}}.
$$
\qed

\begin{lemma}
                        \label{lemma 8.17.1}
For $m=0,1,...$, $n_{1}\geq n_{0}$,
$|x_{0}|\leq R_{n_{1}}$, and Borel $f(t,x)\geq0$ vanishing for $|x|\geq R_{n_{1}+m}$ it holds that
$$
I:=E\int_{0}^{\infty}e^{-\lambda_{0}t}
f(t,x_{t})\,dt\leq N\|f\|_{L_{(p,q)}},
$$
where $N$ depends only on $d,\delta,p,q,R_{n_{1}+m+1},R_{n_{1}}$.
\end{lemma}

Proof. Introduce recursively $\gamma^{0}=0$, $\tau^{1}$
as the first exit time of $x_{t}$ from
$B_{R_{n_{1}+m+1}}$, $\gamma^{1}$ as the first
time after $\tau^{1}$ when $x_{t}$ hits
$\bar B_{R_{n_{1}+m}}$, $\tau^{n}$ as the first 
time after $\gamma^{n-1}$ when $x_{t}$ exits from
$ B_{R_{n_{1}+m+1}}$, $\gamma^{n}$ as the first
time after $\tau^{n}$ when $x_{t}$ hits
$\bar B_{R_{n_{1}+m+1}}$. Then
$$
I=E\sum_{n=0}^{\infty}E_{\cF_{\gamma^{n}}}
\int_{\gamma^{n}}^{\tau^{n+1}}
e^{-\lambda_{0}t}
f(t,x_{t})\,dt
$$
$$
=E\sum_{n=0}^{\infty}e^{-\lambda_{0}\gamma^{n}}E_{\cF_{\gamma^{n}}}
\int_{0}^{\tau^{n+1}-\gamma^{n}}
e^{-\lambda_{0}t}
f(\gamma^{n}+t,x_{\gamma^{n}+t})\,dt.
$$
Here the interior conditional expectations are
dominated by $N\|f\|_{L_{(p,q)}}$ owing to Theorem \ref{theorem 7.29.1} (and Corollary
\ref{corollary 8.16.1}). Furthermore,
$$
E e^{-\lambda_{0}\gamma^{n}}=
E e^{-\lambda_{0}\gamma^{n-1}}E_{\cF_{\tau^{n}}}e^{-\lambda_{0}(\gamma^{n}-\tau^{n})},
$$
where the conditional expectation is dominated
by $\nu_{0}^{\rho^{2}}$, with $\rho=\rho _{n_{1}+m } $, in light of Lemma
\ref{lemma 8.13.2}. Hence,
$$
E e^{-\lambda_{0}\gamma^{n}}\leq \nu_{0}^{\rho^{2}}
E e^{-\lambda_{0}\gamma^{n-1}}\leq \nu_{0}^{ \mu(n_{1},m)},
$$
and our assertion follows. \qed

\mysection{Proof of Theorem \ref{theorem 6.19.01}
and \ref{theorem 9.19.1}}  

First, we deal with Theorem \ref{theorem 6.19.01}
and  make a crucial step proving
and analog of Theorem \ref{theorem 9.6.4}
when Assumption \ref{assumption 7.28.1} is replaced
with \eqref{8.13.1}.

\begin{theorem}
                       \label{theorem 8.17.1}
In our new situation Theorem \ref{theorem 9.6.4} still holds true
if we suppose that \eqref{8.13.1}
is satisfied instead of Assumption \ref{assumption 7.28.1}, which will change Assumption
\ref{assumption 8.16.1} accordingly.
\end{theorem}

We need the following
in which $\cB(\bR)$ is the Borel $\sigma$-field
on $\bR$.
\begin{lemma}
                        \label{lemma 8.17.3}
Let $\alpha\in(0,1)$ and let $\xi(t)\geq0,t\in\bR$, be a $\cF\otimes \cB(\bR)$-measurable function such that for any $n\ge1$
and $t_{1},...,t_{n}\in[0,1]$ we have
$$
P(\max_{i\leq n}\xi(t_{i})\geq1)\leq \alpha.
$$
Then
$$
P(\esssup_{[0,1]} \xi(t)\geq1)\leq \alpha.
$$
\end{lemma}

Proof. We may assume that $\xi(t)\leq2$. Let $\kappa_{n}(t)=2^{-n}\lfloor t2^{n}\rfloor $,
$n=1,2,...$,
and for $s\in[0,1]$, $p\in(1,\infty)$ consider the functions
$$
I_{n,p}(s)=I_{n,p}(\omega,s) =\Big(\int_{0}^{1}
\xi^{p}(\kappa_{n}(t+s)-s)\,dt\Big)^{1/p}.
$$ 
By taking into account that $\kappa_{n}$ is piece-wise constant, one easily sees that
$I_{n}(\omega,s)$ is measurable with respect to
$(\omega,s)$ and
\begin{equation}
                              \label{8.18.2}
P(I_{n,p}(s)\geq1)\leq\alpha.
\end{equation}
 By one of   Doob's theorem, for each $\omega$,
\begin{equation}
                              \label{8.18.1}
I_{n,p}(s)\to\Big(\int_{0}^{1}\xi^{p}(t)
\,dt\Big)^{1/p}=:I_{p}
\end{equation}
as $n\to\infty$ for almost all $s\in[0,1]$. Then by Fubini's theorem
there exists $s\in[0,1]$ such that \eqref{8.18.1} hods almost surely. This and \eqref{8.18.2} imply that $P(I_{p}\geq 1)\leq\alpha$.
Now it only remains to recall that
$$
I_{p}\uparrow \esssup_{[0,1]} \xi(t)
$$
as $p\to\infty$ for any $\omega$. \qed

{\bf Proof of Theorem \ref{theorem 6.19.01}}. First we almost literally follow the proof of Theorem \ref{theorem 9.6.4} and only discuss
the necessary changes. In Step 1 we only need
to replace \eqref{8.16.6} with \eqref{8.17.3}
written for $x^{(n)}_{s}$. Step 2 needs no changes. In Step 3 we first observe that
owing to \eqref{8.17.3} and Lemma \ref{lemma 8.17.3}
$$
\lim_{\rho\to\infty}\sup_{n\geq0}
P(\esssup_{s\leq T}|x^{(n)}_{s}|\geq\rho)=0
$$
for any $T\in(0,\infty)$.
Then we
change the treatment
of $I_{n}$ writing $I_{n}\leq I_{n,\rho}+J_{n,\rho}$,
where
$$
I_{n,\rho}=P\Big(\Big|\int_{0}^{T}(\sigma-\alpha)(t^{(n)}_{0}+s,\hat x_{s}^{(n)})
I_{|\hat x_{s}^{(n)}|\leq\rho}\,d\hat w^{(n)}_{s}
  \Big|\geq \varepsilon/2
\Big),
$$
\begin{equation}
                              \label{8.18.3}
J_{n,\rho}=P(\esssup_{s\leq T}|x^{(n)}_{s}|\geq\rho).
\end{equation}
For any $\rho$ the term
$I_{n,\rho}$ can be made as small as we like
on account of choosing $\alpha$ appropriately.
This is shown as in Step 3 on the basis of
Lemma \ref{lemma 8.17.1} in place of \eqref{4.26.1}.

Treating the second integral in \eqref{5.12.1}
 also requires only a small change. Observe that
$$
P\Big(\int_{0}^{T}|(b^{(n)}-\beta)
(t^{(n)}_{0}+s,
\hat x^{(n)}_{s})|\,ds>2\varepsilon\Big)
$$ 
$$
\leq P\Big(\int_{0}^{T}|(b^{(n)}-\beta)
(t^{(n)}_{0}+s,
\hat x^{(n)}_{s})|I_{|\hat x_{s}^{(n)}|\leq\rho}\,ds>2\varepsilon\Big)+J_{n,\rho}
$$
and one can repeat the same
as above about this estimate.

This leads to an obvious counterpart
of Theorem \ref{theorem 9.6.4} in our new setting.  
Now Theorem \ref{theorem 6.19.01}  
is proved by repeating word for word the
corresponding   part of the
proof of Theorem 1.9.1 of \cite{Kr_26} using
Theorem \ref{theorem 8.17.1} in place of Theorem 1.6.1 of \cite{Kr_26}.
Of course, as the approximating $b^{(n)}$ we take $b\zeta_{n}$,
where $\zeta_{n}$ are cut-off functions, which allows us
to apply Theorem \ref{theorem 6.19.10}. \qed

From the above proof it is clear how to treat the case
when $b$ is represented as the sum of few terms
as in Theorem \ref{theorem 9.19.1}. We leave the details
to the reader.

\mysection{An example}

In the one space dimension for $\alpha>0$ define $b(x)= x_{+}\ln^{1+\alpha}(1+|x|)$ and consider the equation
\begin{equation}
                           \label{8.9.1}
x_{t}=2+w_{t}+\int_{0}^{t}b(x_{s})\,ds.
\end{equation}
Set
$y_{t}=x_{t}-w_{t} $, so that
$$
y_{t}=2+\int_{0}^{t}b(y_{s}+w_{s})\,ds.
$$
Define
$$
T=\int_{2}^{\infty}\frac{1}{b(x-1)}\,dx.
$$
Observe that, if the event $\{\min_{[0,T]}w_{s}>-1\}$
occurs,
then on $[0,T]$
$$
y_{t}\geq 2+\int_{0}^{t}b(y_{s}-1)\,ds.
$$ 
It follows that $y_{t}\geq z_{t}$, where
$z_{t}$ is the solution of
$$
z_{t}= 2+\int_{0}^{t}b(z_{s}-1)\,ds
$$
(cf. the argument on page 79 of \cite{Kr_77}).
Obviously, $z_{t}>2$, so that $b(z_{s}-1)\geq b(1)>0$ and $z_{t}$ satisfies
$$
\int_{z_{t}}^{\infty}\frac{1}{b(x-1)}\,dx=
T-t.
$$
 It follows that $y_{t}\geq z_{t}\to\infty$
as $t\uparrow T$, so that \eqref{8.9.1}
does not have solutions defined for all $t\geq0$.

Now let us see how much condition \eqref{8.13.1} is violated. For $x\geq1$ we set $\hat \rho(x)=K_{0}x^{-1}$. In that case, for $x\geq1$
$$
\dashnorm b\|^{p}_{L_{(p,q)}(C_{  
\hat\rho}(t,x))}=(1/2)\hat \rho^{-1}(x)
\int_{x-\hat\rho(x)}^{x+\hat\rho(x)}b^{p}(y)\,dy\geq b^{p}(x-\hat\rho(x))
$$
implying that
$$
\dashnorm b\| _{L_{(p,q)}(C_{  
\hat\rho}(t,x))}\geq b (x-\hat\rho(x))
=\hat\rho^{-1}(x)\hat b(x),
$$
where $\hat b(x)=K_{0}^{-1}b (x-\hat\rho(x))/(x
\ln x)
\sim K_{0}^{-1}\ln^{ \alpha}x$ as $x\to \infty$. We see
that in condition \eqref{8.13.1} we cannot allow
$\hat b$ to grow to infinity even with such a slow rate. 

{\bf Declarations}. 
No funds, grants, or other support was received.
The author has no relevant financial or non-financial interests to disclose.
The manuscript contains no data.

\end{document}